\documentclass[12pt,reqno]{amsart}
\usepackage{enumerate, latexsym, amsmath, amsfonts, amssymb, amsthm, graphicx, color}
\def\pmod #1{\ ({\rm{mod}}\ #1)}
\def\Z{\Bbb Z}
\def\N{\Bbb N}

\def\Q{\Bbb Q}

\def\R{\Bbb R}

\def\bg{\bigg}
\def\({\bg(}
\def\){\bg)}

\def\ord{{\rm ord}}

\def\ve{\varepsilon}

\theoremstyle{plain}
\newtheorem{theorem}{Theorem}

\newtheorem{lemma}{Lemma}
\newtheorem{corollary}{Corollary}
\newtheorem{conjecture}{Conjecture}
\newtheorem{proposition}{Proposition}
\theoremstyle{definition}

\theoremstyle{remark}
\newtheorem{remark}{Remark}

\makeatletter
\@namedef{subjclassname@2010}{%
  \textup{2010} Mathematics Subject Classification}
\makeatother
 \vspace{4mm}

\begin{document}
 \baselineskip=17pt
\hbox{}
\medskip
\title[On the sums of three generalized polygonal numbers]
{On the sums of three generalized polygonal numbers}
\date{}
\author[Hai-Liang Wu and Hao Pan] {Hai-Liang Wu and Hao Pan}

\thanks{
\newline \indent Supported by the National Natural Science
Foundation of China (Grant No. 11571162).}

\address {(Hai-Liang Wu)  Department of Mathematics, Nanjing
University, Nanjing 210093, People's Republic of China}
\email{\tt whl.math@smail.nju.edu.cn}

\address {(Hao Pan) Department of Mathematics, Nanjing
University of Finance $\&$ Economics, Nanjing 210023, People's Republic of China}
\email{{\tt haopan79@zoho.com}}

\begin{abstract}
For each natural number $m\ge 3$, let $P_m(x)$ denote the generalized $m$-gonal number $\frac{(m-2)x^2-(m-4)x}{2}$ with $x\in\Z$. In this paper, with the help of the congruence theta function, we establish conditions on $a$, $b$, $c$ for which the sum $P_a(x)+P_b(y)+P_c(z)$ represents all but finitely many positive integers.
\end{abstract}

\maketitle

\section{Introduction}
\setcounter{lemma}{0}
\setcounter{theorem}{0}
\setcounter{corollary}{0}
\setcounter{remark}{0}
\setcounter{equation}{0}
\setcounter{conjecture}{0}
\setcounter{proposition}{0}

For each integer $m\ge 3$, generalized $m$-gonal numbers are those numbers of the form
$P_m(x)=\frac{(m-2)x^2-(m-4)x}{2}$ with $x\in\Z$. In 1638, Fermat first claimed that each natural number can
be written as the sum of 3 triangular numbers, 4 squares, 5 pentagonal numbers, and in general $m$ $m$-gonal
numbers (cf. \cite[Chapter I]{N}), and this claim was proved completely by Cauchy in 1813. In this direction,
many mathematicians have studied the problems of almost universal quadratic polynomials
(a quadratic polynomial is said to be almost universal, if
it represents all but finitely many positive integers over $\Z$). Guy \cite{Guy} studied the minimal number
$r_m\in\N$ chosen such that every natural number can be written as the sum of $r_m$ generalized $m$-gonal numbers. Guy pointed out that for sufficiently large integer $n$, one could likely
represent $n$ with significantly fewer generalized $m$-gonal numbers.
Shimura \cite{Shimura2} investigated the inhomogeneous quadratic polynomials and showed an explicit
formula for the number of representations of an integer as the sum of $n$ triangular numbers for each
$n$ in the range $2\le n\le8$.
Kane and Sun \cite{KS} obtained a classification of almost universal weighted sums of triangular numbers and more generally weighted mixed ternary sums of triangular
and square numbers, this classification was later completed by Chan and Oh \cite{WKCOH}
and Chan and Haensch \cite{WKCANNA}. A. Haensch \cite{ANNA} investigated the almost universal ternary quadratic polynomials with odd prime power conductor.

Recently, A. Haensch and B. Kane
\cite{AK} proved that for each $m\ge 3$ satisfying $m\not\equiv 2\pmod 3$ and $4\nmid m$, every sufficiently
large natural number can be written as the form $P_m(x)+P_m(y)+P_m(z)$ with $x,y,z\in \Z$. Obviously, the most general case of this old problem is to study the ternary sums of the type
\begin{equation*}
aP_l(x)+bP_m(y)+cP_n(z),
\end{equation*}
where $a,b,c\in \Z^+$ and $l,m,n\ge 3$.

Before the statement of the main results of this paper, we first see a concrete example, Sun \cite{Sun} proved that the polynomial $P_3(x)+P_4(y)+P_5(z)$ can represent all natural numbers over $\Z$ (Sun even showed that $y$ can be chosen be an even integer). Moreover, Sun \cite{Sun} listed all the possible
candidates of the form
$aP_l(x)+bP_m(y)+cP_n(z)$ with $a,b,c\ge 1$ that represent all the natural numbers over $\Z$.

Recently, via careful computation, Sun made the following conjecture.

\begin{conjecture}\label{Con1.1}
For any integer $m\ge 3$, $P_m(x)+P_{m+1}(y)+P_{m+2}(z)$ is almost universal.
\end{conjecture}
Furthermore, Sun even conjectured that every sufficiently large integer can be represented as the form
$P_m(x)+P_{m+1}(y)+P_{m+2}(z)$
with $x,y,z\in\N$.

Motivated by the above, in this paper, we concentrate on the sums of the form
\begin{equation*}
\mathcal{F}_{a,b,c}(x,y,z):= P_a(x)+P_b(y)+P_c(z).
\end{equation*}
Now, we state the main results of this paper. Set
\begin{align*}
l_n=&2^{3-\delta}(a-2)(b-2)(c-2)n,
\\N=&2^{(-\delta+2)/2}(a-2)(b-2)(c-2),
\\v_{a,b,c}=&2^{-\delta}((a-4)^2(b-2)(c-2)+(a-2)(b-4)^2(c-2)+(a-2)(b-2)(c-4)^2),
\\\mathcal{S}_{a,b,c}^t=&\{n\in\N: l_n+v_{a,b,c}=tr^2\ \text{for some $r\in\N$}\},
\\\mathcal{N}_{a,b,c}=&\{n\in\N: \text{$n$ can not be represented by}\ \mathcal{F}_{a,b,c}(x,y,z)\},
\end{align*}
where $t\in\N$ and
$$\delta=\begin{cases}0&\mbox{if}\ 2\nmid(a,b,c),\\2&\mbox{if}\ 2\mid(a,b,c).\end{cases}$$

\begin{theorem}\label{Thm1.1}
Let notations be as above, if at least one of $a,b,c$ is not divisible by $4$, then we have the following results.

{\rm (1)} Any sufficiently large integer $n\not\in\mathcal{S}_{a,b,c}$ can be represented by $\mathcal{F}_{a,b,c}(x,y,z)$,
where
$$\mathcal{S}_{a,b,c}=\bigcup_{\substack{t\mid N\\ \text{t squarefree}}}
\mathcal{S}_{a,b,c}^t.$$

{\rm (2)} Suppose that the greatest common divisor of any two of $a-2,b-2,c-2$ is of the form
$2^r$ with $r\in\N$, then $\mathcal{F}_{a,b,c}(x,y,z)$ is almost universal if all of the following conditions
are satisfied:

{\rm (i)} There exists an odd prime $p$ such that one of $a-2,b-2,c-2$ is divisible by $p$ and the product
of the remaining two numbers is not a quadratic residue modulo $p$.

{\rm (ii)} When $2\mid abc$, then there exists an odd prime $q$ such that one of $a-2,b-2,c-2$ is divisible of
$q$ and the product of $2$ and the remaining two numbers is not a quadratic residue modulo $q$.

\end{theorem}

\begin{remark}\label{Rem1.1}
In the proof of Theorem \ref{Thm1.1}, we actually offer an explicit method to show that $\mathcal{F}_{a,b,c}(x,y,z)$ is almost universal for some particular $(a,b,c)\in\Z^3$.

We will see in Lemma \ref{Lem2.1} that if one of $a,b,c$ is not divisible by $4$, then
there is no local obstruction to $\mathcal{F}_{a,b,c}(x,y,z)$ being almost universal. However,
there is a local obstruction to $\mathcal{F}_{a,b,c}$ being almost universal if
$a\equiv b\equiv c\equiv 0\pmod 4$.
\end{remark}

For each prime $p\equiv\pm1\pmod8$, $2$ is a quadratic residue modulo $p$,
by virtue of Theorem \ref{Thm1.1}, we can easily get the following results.

\begin{corollary}\label{(i)}
Let notations and assumptions be as in {\rm (2)} of Theorem \ref{Thm1.1}, if there is a prime $p$ with
$p\equiv\pm1\pmod8$ satisfying the condition {\rm (i)} of {\rm (2)} of Theorem \ref{Thm1.1}, then
$\mathcal{F}_{a,b,c}$ is almost universal.
\end{corollary}

\begin{proposition}\label{Pro1.1}
Let notations be as above, if at least one of $a,b,c$ is not divisible by $4$ and the greatest common divisor of any two of $a-2,b-2,c-2$ is of the form $2^r$ with $r\in\N$, then we have the following results.

{\rm (1)} Suppose that $a,b,c$ are distinct modulo $3$,
then $\mathcal{F}_{a,b,c}(x,y,z)$ is almost universal, if one of the following conditions is satisfied:

{\rm (i)} $2\nmid abc$,

{\rm (ii)} one of $a,b,c$ is congruent to $2$ modulo $4$ and the remaining two
numbers are odd or are all
divisible by $4$,

{\rm (iii)} one of $a,b,c$ is divisible by $4$ and the remaining two odd integers of $a,b,c$ have same residues modulo $4$.

{\rm (2)} Suppose that one of the above conditions {\rm (i), (ii), (iii)} is satisfied, if $a\equiv b\equiv c\not\equiv 2\pmod3$,
then any sufficiently large integer $n$ with $n\equiv a+1\pmod3$ can be represented by $\mathcal{F}_{a,b,c}(x,y,z)$.
\end{proposition}

Analogous to the above concrete example $P_3(x)+P_4(y)+P_5(z)$, we have the following corollary:

\begin{corollary}\label{m,m+1,m+2}
For each integer $m\ge 3$, we have

{\rm (1)} $P_m(x)+P_{m+1}(y)+P_{m+2}(z)$ is almost universal if one of the following conditions is satisfied:

{\rm (i)} $m\equiv 1\pmod4$.

{\rm (ii)} There exists an odd prime $q$ such that one of $m-2,m-1,m$ is divisible of
$q$ and the product of $2$ and the remaining two numbers is not a quadratic residue modulo $q$.

{\rm (2)} $\mathcal{N}_{m,m+1,m+2}\setminus \mathcal{S}_{m,m+1,m+2}^2$ is a finite set.
\end{corollary}

For each integer $k\ge 1$, the $k$-th Fermat's number is defined to be
$F_k:=2^{2^k}+1$, it is well known that any two different Fermat's numbers are coprime. On the other hand,
in mathematics, a Mersenne prime is a prime number the form $M_p:=2^p-1$ for some prime $p$. It is well known
that the largest known prime number $2^{77232917}-1$ is a Mersenne prime.
Then by (2) of Proposition \ref{Pro1.1}, we can easily get the following results
involving Fermat's numbers and Mersenne primes.

\begin{corollary}\label{Fermat and Mersenne}
Let notations be as above, we have the following results:

{\rm (1)} For any sufficiently large integer $n$ with $n\equiv 2\pmod 3$, we can write $n$ as the form
\begin{equation*}
P_{F_k+2}(x)+P_{F_l+2}(y)+P_{F_m+2}(z)
\end{equation*}
with $x,y,z\in\Z$, where $k,l,m$ are pairwisely distinct positive integers.

{\rm (2)} For any sufficiently large integer $n$ with $n\equiv 1\pmod 3$, we can write $n$ as the form
\begin{equation*}
P_{M_p+2}(x)+P_{M_q+2}(y)+P_{M_r+2}(z)
\end{equation*}
with $x,y,z\in\Z$, where $p,q,r$ are distinct odd primes.
\end{corollary}

Let $\alpha,\beta,\gamma$ be positive odd integers, and let $k,m,l$ be natural numbers satisfying
$l\ge k\ge m\ge2$, further we put $2^k\alpha-2=2\ve$, $2^l\beta-2=2\eta$ and $2^m\gamma-2=2\mu$.
Then we have the following results.

\begin{corollary}\label{2^k}
Let notations be as above, suppose that $\alpha,\beta,\gamma$ are pairwisely coprime, then
$P_{2^k\alpha+2}(x)+P_{2^l\beta+2}(y)+P_{2^m\gamma+2}(z)$ is almost universal, if one of the following conditions is
satisfied:

{\rm (i)} $\alpha,\beta,\gamma$ are distinct modulo $3$ and $k=l=m$,

{\rm (ii)} $\alpha,\beta,\gamma$ are distinct modulo $3$, $k>l$ and $k\equiv l\equiv m\pmod2$,

{\rm (iii)} $3\mid\gamma$, $\alpha\equiv\beta\equiv1\pmod{12}$, $k=l>m$, $k-m>1$ and $k\not\equiv m\pmod2$.
\end{corollary}

We now give an outline of this paper, in Section 2,
we will give a brief overview of the theory of ternary quadratic forms and
congruence theta function which we need in our proofs and prove some useful lemmas, finally we shall prove the
main results of this paper in Section 3.
\maketitle
\section{Some Lemmas}
\setcounter{lemma}{0}
\setcounter{theorem}{0}
\setcounter{corollary}{0}
\setcounter{remark}{0}
\setcounter{equation}{0}
\setcounter{conjecture}{0}

In this section, we first introduce some necessary objects used in our proofs and then clarify the connections
between congruence theta function and representation of natural numbers by $\mathcal{F}_{a,b,c}(x,y,z)$.
\subsection{Shifted Lattice Theory}
In this paper, we adopt the language of quadratic space and lattices as in \cite{Ki}, given a
positive definite quadratic space $(V,B,Q)$, let $L$ be a lattice on $V$ and $A$ be a symmetric matrix,
we denote $L\cong A$ if $A$ is the gram matrix for $L$ with respect to some basis. In particular,
an $n\times n$ diagonal matrix with $a_1,...,a_n$ as the diagonal entries is written as $\langle a_1,...,a_n\rangle$. For each place $p$ of $\Q$, we define the localization of $L$ by
$L_p=L\otimes_{\Z}\Z_p$, especially, $L_{\infty}=V\otimes_{\Q}\R$.

For a vector $v\in V$, we have a lattice coset $L+v$, similar to the definitions of
$gen(L),\ spn(L),\ cls(L)$ (cf. \cite[pp. 132--133]{Ki}), we have the following definitions that
originally appear in \cite{V}.

The class of $L+v$ is defined as
\begin{align*}
cls(L+v):= \text{the orbit of $L+v$ under the action of $O(V)$},
\end{align*}

the spinor genus of $L+v$ is defined as
\begin{align*}
spn(L+v):= \text{the orbit of $L+v$ under the action of $O(V)O_{\mathbb{A}}^{'}(V)$},
\end{align*}

the genus of $L+v$ is defined as
\begin{align*}
gen(L+v):= \text{the orbit of $L+v$ under the action of $O_{\mathbb{A}}(V)$},
\end{align*}
the definitions of the adelization groups $O(V)$, $O_{\mathbb{A}}^{'}(V)$, $O_{\mathbb{A}}(V)$ may be
found in \cite[p. 132]{Ki}.

We now consider the representation of a natural number $n$ by $\mathcal{F}_{a,b,c}(x,y,z)$, by completing the
square, it is easy to see that $n$ can be represented by $\mathcal{F}_{a,b,c}(x,y,z)$ if and only if
$l_n+v_m$ can be represented by the lattice coset $L+v$, where
\begin{equation}\label{lattice}
L=\begin{cases}4(a-2)(b-2)(c-2)\langle a-2,\ b-2,\ c-2\rangle&\mbox{if}\ 2\nmid(a,b,c),
\\(a-2)(b-2)(c-2)\langle a-2,\ b-2,\ c-2\rangle&\mbox{if}\ 2\mid(a,b,c)\end{cases}
\end{equation}
with respect to the orthogonal basis $\{e_1,e_2,e_3\}$ and
\begin{equation}\label{vector}
v=-\frac12(\frac{a-4}{a-2}e_1+\frac{b-4}{b-2}e_2+\frac{c-4}{c-2}e_3).
\end{equation}
We also set
\begin{equation}\label{space}
V=\Q e_1\perp\Q e_1\perp\Q e_3.
\end{equation}
Let notations be as above, we have the following lemma involving local representations.

\begin{lemma}\label{Lem2.1}
If one of $a$, $b$, $c$ is not divisible by $4$, then for each prime $p$,
$l_n+v_{a,b,c}$ can be represented by $L_p+v$.
\end{lemma}
\begin{proof}
We divide our proof into three cases.

{\it Case} 1. $p\nmid 2(a-2)(b-2)(c-2)$.

Then $L_p+v=L_p$ and $L_p$ is unimodular, thus by \cite[Lemma 3.4]{C},
$$L_p\cong\begin{pmatrix} 0 & 1 \\1 & 0 \end{pmatrix}\perp\langle-d(L)\rangle,$$
where $d(L)$ denotes the discriminant of $L$, clearly, $L_p$ can represent all $p$-adic integers over $\Z_p$.

{\it Case} 2. $p\mid 2(a-2)(b-2)(c-2)$ and $p>2$.

Without loss of generality, we may assume that $p\mid a-2$. We now show that $P_a(x)$ can represent all $p$-adic integers over $\Z_p$. Put $a-2=p^k\ve$ and $a-4=\gamma$, where $\ve$, $\gamma\in\Z_p^{\times}$
(where $\Z_p^{\times}$ denotes the set of all $p$-adic units). For any $n\in\Z_p$,
\begin{equation}\label{equation p}
n=\frac{(a-2)x^2-(a-4)x}{2}=\frac{p^k\ve x^2-\gamma x}{2},
\end{equation}
then (\ref{equation p}) has a solution in the algebraic closure of $\Q_p$,
\begin{equation*}
x=\frac{\gamma\pm\sqrt{\gamma^2+8p^k\ve n}}{2p^k\ve},
\end{equation*}
by Local Square Theorem (cf. \cite[Lemma 1.6]{C}), there exists a $u\in\Z_p^{\times}$ such that
$\gamma^2+8p^k\ve n=\gamma^2u^2$, then we have $\gamma^2\equiv \gamma^2u^2\pmod {p^k}$, this implies that
$u\equiv\pm1\pmod {p^k}$, without loss of generality, we may write $u=1-p^kw$ with $w\in\Z_p$, then
(\ref{equation p}) has a solution
\begin{equation*}
x=\frac{\gamma-\gamma(1-p^kw)}{2p^k\ve}=\frac{\gamma w}{2\ve}\in\Z_p.
\end{equation*}

{\it Case} 3. $p=2$.

We claim that when $m$ is odd or $m\equiv 2\pmod 4$, $P_m(x)$ can represent all $2$-adic integers over $\Z_2$, since one of $a$, $b$, $c$ is not divisible by $4$, then one can easily get the desired result if the claim is true. We now prove the claim, when $m$ is odd, set $m-2=\alpha$, $m-4=\beta$, where $\alpha$, $\beta\in\Z_2^{\times}$. For each $n\in\Z_2$,
\begin{equation}\label{equation 2}
n=\frac{(m-2)x^2-(m-4)x}{2}=\frac{\alpha x^2-\beta x}{2},
\end{equation}
then (\ref{equation 2}) has a solution in the algebraic closure of $\Q_2$,
\begin{equation*}
x=\frac{\beta\pm\sqrt{\beta^2+8\alpha n}}{2\alpha}=\frac{1\pm\sqrt{1+8\ve n}}{2\mu},
\end{equation*}
where $\ve=\alpha/\beta^2$ and $\mu=\alpha/\beta$. By Local Square Theorem, we may write
$1+8\ve n=(1-2\gamma)^2$ with $\gamma\in\Z_2$, then $x=\gamma/\mu\in\Z_2$ is a solution of (\ref{equation 2}).
Assume now $m\equiv 2\pmod 4$, set $m-2=2^{k+1}\alpha$, and $m-4=2\beta$, where $\alpha$, $\beta\in\Z_2^{\times}$ and $k\ge 1$. For each $n\in\Z_2$, the equation $n=P_m(x)$ has a solution in the algebraic closure of $\Q_2$ as follows.
\begin{equation*}
x=\frac{\beta\pm\sqrt{\beta^2+2^{k+2}\alpha n}}{2^{k+1}\alpha}=\frac{1\pm\sqrt{1+2^{k+2}\ve n}}{2^{k+1}\mu},
\end{equation*}
where $\ve=\alpha/\beta^2$ and $\mu=\alpha/\beta$. since $k\ge1$, by Local Square Theorem, we may write
$1+2^{k+2}\ve n=(1-2^s\gamma)^2=1-2^{s+1}\gamma+2^{2s}\gamma^2$ with $\gamma\in\Z_2^{\times}$ and $s\ge 1$.
If $s>1$, then
$1+2^{k+2}\ve n=1-2^{s+1}\gamma(1-2^{s-1}\gamma)$, this implies that $s\ge k+1$, thus the equation has a solution $x=2^{s-(k+1)}\gamma/\mu\in\Z_2$. If $s=1$, then $1+2^{k+2}\ve n=1-2^2\gamma(1-\gamma)$, this implies that $\ord_2(1-\gamma)\ge k$ (where $\ord_p$ denotes the $p$-adic order), thus the equation has a solution
\begin{equation*}
x=\frac{1+(1-2\gamma)}{2^{k+1}\mu}=\frac{1-\gamma}{2^k\mu}\in\Z_2.
\end{equation*}

In view of the above, we complete the proof.
\end{proof}
\begin{remark}\label{Rem2.1}
When $a\equiv b\equiv c\equiv 4s\pmod 8$, where $s\in\{0,1\}$, one may easily verify that $\#\{P_a(x)+P_b(y)+P_c(z)+8\Z:
x,y,z\in\Z\}<8$ ($\#S$ denotes the cardinality of a finite set $S$).
When $a\equiv b\equiv c\equiv 0\pmod4$ and two of $a,b,c$ have distinct residues modulo $8$, then one may easily verify that $\#\{P_a(x)+P_b(y)+P_c(z)+16\Z:
x,y,z\in\Z\}<16$.
Thus there is a local obstruction to $\mathcal{F}_{a,b,c}$ being almost universal.
\end{remark}

\subsection{Congruence Theta Functions}
In \cite{Shimura1} Shimura investigated the modular forms of half-integral weights and explicit
transformation formulas for the theta functions. Let $n$ and $N$ be positive integers, let $A$ be an $n\times n$ integral positive definite symmetric matrix such that $NA^{-1}$ is an integral matrix. Further set $l$ be a non-negative integer, and $P$ be a spherical function of order $l$ with respect to $A$, given a column vector $h\in\Z^n$ with $Ah\in N\Z^n$, we consider the following theta function (\cite[2.0]{Shimura1}), with variable $z\in\mathbb{H}$ (where $\mathbb{H}$ denotes the upper half plane),
\begin{equation}\label{theta function}
\theta(z,h,A,N,P)=\sum_{m\equiv h\pmod N}P(m)\cdot e(z\cdot m^tAm/2N^2),
\end{equation}
where the summation runs over all $m\in\Z^n$ such that $m\equiv h\pmod N$, $m^t$ denotes the transposed matrix of $m$ and $e(z)=e^{2\pi iz}.$ We usually call the theta function of this type {\it congruence theta function}. When $P=1$,
we simply write $\theta(z,h,A,N)$ instead of $\theta(z,h,A,N,P)$.
By \cite[Propostion 2.1]{Shimura1}, if the diagonal elements of $A$ are even, then
$\theta(z,h,A,N,P)$ is a modular form of weight $n/2$ on $\Gamma_1(2N)$ with some multiplier.

Congruence theta functions have close connection with the representations of natural numbers by quadratic forms with some additional congruence conditions. To see this, we turn to the representation of natural number $n$ by $\mathcal{F}_{a,b,c}(x,y,z)$.
Let $L$ and $v$ be as in (\ref{lattice}) and (\ref{vector}) respectively. We define
\begin{equation}\label{theta F}
\theta_{L+v}(z):=\sum_{n\ge0}r_{L+v}(n)e(nz),
\end{equation}
where $z\in\mathbb{H}$, and $r_{L+v}(n):=\#\{x\in L: Q(x+v)=n\}$. On the other hand,
put
\begin{align*}
N&=2^{\delta}(a-2)(b-2)(c-2),
\\A&=2^{\delta}\langle a-2,\ b-2,\ c-2\rangle,
\\h&=-2^{\delta-1}((a-4)(b-2)(c-2),(a-2)(b-4)(c-2),(a-2)(b-2)(c-4))^t,
\end{align*}
where
\begin{equation}\label{delta}
\delta=\begin{cases}1&\mbox{if}\ 2\nmid(a,b,c),\\0&\mbox{if}\ 2\mid(a,b,c).\end{cases}
\end{equation}
Then one may easily verify that $\theta_{L+v}(z)=\theta(2Nz,h,A,N)$, thus $\theta_{L+v}(z)$ is a modular form
of weight $3/2$ on $\Gamma_1(4N^2)$. In addition, we need the following unary congruence theta functions,
let $N=2^{\delta}(a-2)(b-2)(c-2)$ with $\delta$ as in (\ref{delta}) and $t$ be a squarefree factor of $N$, further set $A=(N/t)$ and $P(n)=n$, we define
\begin{equation}\label{unary}
u_{h,t}(z):=\theta(2Nz,h,A,N/t,P)=\sum_{r\equiv h\pmod {N/t}}re(tr^2z),
\end{equation}
where $h$ be chosen modulo $N/t$ and the summation runs over all integers $r$ such that $r\equiv h\pmod {N/t}$. Then $u_{h,t}(z)$ is a modular form of weight $3/2$ on $\Gamma_1(4N^2)$ with the same multiplier
of $\theta_{L+v}(z)$.

As in \cite{WR}, we may decompose $\theta_{L+v}(z)$ into the following three parts
\begin{equation}\label{decomposition}
\theta_{L+v}(z)=\mathcal{E}(z)+\mathcal{U}(z)+f(z),
\end{equation}
where $\mathcal{E}(z)$ is in the space generated by Eisenstein series, $\mathcal{U}(z)$ is in the space generated by unary theta functions defined in (\ref{unary}), and $f(z)$ is a cusp form which is orthogonal to those unary theta functions.

We first study the function $\mathcal{E}(z)$, Shimura \cite{Shimura2} proved that
$\mathcal{E}(z)$ is the weighted average of representations by the members of the genus of $L+v$, and simultaneously the product of local densities, thus if $n$ can be represented by $L+v$ locally, and $n$ has bounded divisibility at each anisotropic prime of $V$, then its $n$-th fourier coefficient $a_E(n)\gg n^{1/2-\ve}$. By virtue of \cite{Duke}, the $n$-th fourier coefficient of $f(z)$ grows at most like
$n^{3/7+\ve}$. Hence, if $l_n+v_{a,b,c}$ is represented by $L+v$ locally and has bounded divisibility at each anisotropic prime of $V$, and the ($l_n+v_{a,b,c}$)-th fourier coefficient of $\mathcal{U}$ equals zero,
it is easy to see that $\mathcal{F}_{a,b,c}$ is almost universal.

To show that $l_n+v_{a,b,c}$ has bounded divisibility at each anisotropic prime of $V$, we need to introduce some
results involving quadratic forms over local fields (cf. \cite[Chapter 4]{C}).
For a finite prime $p$ and a non-singular ternary quadratic form $g$ over $\Q_p$,
we adopt the definition of Hasse symbol $c_p(g)$ as in \cite[p. 55]{C}
(note that there are several different definitions of the Hasse symbol).
We also need the following result (cf. \cite[Lemma 2.5]{C}).
\begin{lemma}\label{Cassels}
A necessary and sufficient conditions that the non-singular ternary quadratic form $g$ over $\Q_p$ be isotropic is that
$$c_p(g)=(-1,-d(g))_p,$$
where $d(g)$ denotes the discriminant of $g$ and $(\ ,\ )_p$ denotes the Hilbert symbol over $\Q_p$.
\end{lemma}

Let notations be as above, we have the following two lemmas involving bounded divisibility at anisotropic primes.
\begin{lemma}\label{isotropic odd}
$l_n+v_{a,b,c}$ has bounded divisibility at each odd anisotropic prime of $V$.
\end{lemma}
\begin{proof}
When $p\nmid 2(a-2)(b-2)(c-2)$, as in the Case 1 in the proof of Lemma \ref{Lem2.1}, it is easy to see that $V_p$ is isotropic.

When $p\mid (a-2)(b-2)(c-2)$ and $p>2$, we set $k=\ord_p(a-2)$, $l=\ord_p(b-2)$, $m=\ord_p(c-2)$,
without loss of generality, we assume that $k\ge l\ge m$.
We shall divide the remaining proof into following three cases.

{\it Case} 1. $k=l=m$.

In this case, by scaling, $V_p^{1/p^k}$ is unimodular, thus $V_p$ is isotropic.

{\it Case} 2. $k>l$.

One may easily verify that $\ord_p(l_n+v_{a,b,c})=l+m$.

{\it Case} 3. $k=l>m$.

In this case, by scaling, we may simply assume that $k=l>m=0$, put $a-2=p^k\alpha$, $b-2=p^k\beta$,
$c-2=\gamma$, $a-4=\ve$ and $b-4=\eta$, where $\alpha,\ \beta,\ \gamma,\ \ve,\ \eta\in\Z_p^{\times}$.
Since
\begin{equation*}
(a-4)^2(b-2)(c-2)+(a-2)(b-4)^2(c-2)=p^k\gamma(\alpha\eta^2+\beta\ve^2).
\end{equation*}

If $\alpha\eta^2+\beta\ve^2\not\equiv 0\pmod p$, then $\ord_p(l_n+v_{a,b,c})\le k+m$.

If $\alpha\eta^2+\beta\ve^2\equiv 0\pmod p$, then $-\alpha\beta\equiv \beta^2\ve^2/\eta^2\pmod p$,
by Local Square Theorem, we have $-\alpha\beta\in\Q_p^{\times2}$, thus $\langle a-2,\ b-2\rangle$ is
isotropic and hence $V_p$ is also isotropic.

In view of the above, we complete the proof.
\end{proof}

Bounded divisibility at prime $2$ is much more complicated, by Lemma \ref{Lem2.1}, we always assume that
one of $a,\ b,\ c$ is not divisible by $4$, we put
$k=\ord_2(a-2),\ l=\ord_2(b-2),\ m=\ord_2(c-2)$, without loss of generality, we assume that $k\ge l\ge m$,
note that under our assumption, when $2\mid (a,b,c)$, we must have $a\equiv 2\pmod 4$. We further set $a-2=2^k\alpha$, $b-2=2^l\beta$, $c-2=2^m\gamma$ with $\alpha,\ \beta,\ \gamma\in\Z_2^{\times}$, then we have the following results.

\begin{lemma}\label{isotropic 2}
Let notations and assumptions be as above, if at least one of $a,b,c$ is not divisible by $4$, then
we have either $V_2$ is isotropic or $l_n+v_{a,b,c}$ has bounded divisibility at $2$.
\end{lemma}
\begin{proof}
By scaling, for convenience, we still write $V=\langle a-2,b-2,c-2\rangle$ throughout this lemma.
We first consider the case when $2\nmid (a,b,c)$, we shall divide the proof of this part into three cases.

{\it Case} 1.1. $abc\not\equiv0\pmod2$.

It is clear that $l_n+v_{a,b,c}\not\equiv 0\pmod 2$ in this case.

{\it Case} 1.2. $a\equiv0\pmod2$ and $2\nmid bc$.

Assume first that $a\equiv 2\pmod 4$, if $k>2$, it is clear that $\ord_2(l_n+v_{a,b,c})=2$. If $k=2$,
then $a\equiv6\pmod8$, one may easily verify that $v_{a,b,c}\equiv 4\pmod8$, thus $\ord_2(l_n+v_{a,b,c})=2.$

Assume now that $a\equiv 0\pmod4$, if $b\equiv c\pmod 4$, since
\begin{equation*}
(b-4)^2(c-2)+(b-2)(c-4)^2\equiv b+c\equiv 2\pmod 4,
\end{equation*}
we have $\ord_2(l_n+v_{a,b,c})=2$. If $b\not\equiv c\pmod4$, then
\begin{equation*}
c_2(V_2)=(2\alpha,\beta)_2(2\alpha,\gamma)_2(\beta,\gamma)_2=(-1)^{(\alpha+1)/2}.
\end{equation*}
On the other hand,
$$(-1,-d(V_2))_2=-(-1,\alpha)_2=(-1)^{(\alpha+1)/2}.$$
Thus by Lemma \ref{Cassels}, $V_2$ is isotropic.

{\it Case} 1.3. $a\equiv b\equiv0\pmod 2$ and $2\nmid c$.

When $a\equiv b\equiv 0\pmod 4$, then it is clear that $\ord_2(l_n+v_{a,b,c})=2$.

When $a\equiv b\equiv 2\pmod 4$, set $a-4=2\ve,\ b-4=2\eta,\ c-4=\mu$ with
$\ve,\ \eta,\ \mu\in\Z_2^{\times}$.Assume first that $k>l\ge2$, one may easily verify that
$\ord_2(l_n+v_{a,b,c})=2+l$. Assume now that $k=l>4$, since
\begin{equation*}
(a-4)^2(b-2)(c-2)+(a-2)(b-4)^2(c-2)=2^{k+2}\gamma(\alpha\eta^2+\beta\ve^2).
\end{equation*}
If $\alpha\eta^2+\beta\ve^2\not\equiv 0\pmod8$, then it is easy to see that
$\ord_2(l_n+v_{a,b,c})=k+2+\ord_2(\alpha\eta^2+\beta\ve^2)$.

If $\alpha\eta^2+\beta\ve^2\equiv0\pmod8$, then $-\alpha\beta\equiv \beta^2\ve^2/\eta^2\pmod8$,
by Local Square Theorem, we have $-\alpha\beta\in\Q_2^{\times2}$, thus $V_2$ is isotropic.

If $k=l=2$, then $v_{a,b,c}=2^4(\alpha\eta^2\gamma+\beta\ve^2\gamma+\alpha\beta\mu^2)$, it is clear that
$\ord_2(l_n+v_{a,b,c})=4$.

If $k=l=3$, then
$v_{a,b,c}=2^6(\gamma(\frac{\alpha\eta^2+\beta\ve^2}{2})+\alpha\beta\mu^2)$,
Suppose first that $8\nmid \gamma(\frac{\alpha\eta^2+\beta\ve^2}{2})+\alpha\beta\mu^2$, then it is clear that $\ord_2(l_n+v_{a,b,c})<9$. Assume now that $8\mid \gamma(\frac{\alpha\eta^2+\beta\ve^2}{2})+\alpha\beta\mu^2$,
this implies that $\alpha\equiv\beta\pmod4$, we set $\alpha-\beta\equiv 4s\pmod 8$, where $s\in\{0,1\}$,
then we have
\begin{equation*}
c_2(V_2)=(2^3\alpha,2^3\beta)_2(2^3\alpha,\gamma)_2(2^3\beta,\gamma)_2=(-1)^s(-1)^{(\alpha-1)/2}.
\end{equation*}
On the other hand, since $8\mid \gamma(\frac{\alpha\eta^2+\beta\ve^2}{2})+\alpha\beta\mu^2$, we have
$\gamma\equiv 2s-\alpha\pmod4$, thus
$$(-1,-(a-2)(b-2)(c-2))_2=(-1)^s(-1)^{(\alpha-1)/2}.$$
By Lemma \ref{Cassels}, $V_2$ is isotropic.

If $k=l=4$, then
\begin{equation*}
(a-4)^2(b-2)(c-2)+(a-2)(b-4)^2(c-2)=2^6\gamma(\alpha\eta^2+\beta\ve^2),
\end{equation*}
Suppose first that $\alpha\equiv\beta\pmod4$, then $4\nmid\alpha\eta^2+\beta\ve^2$, thus $\ord_2(l_n+v_{a,b,c})=7$. Assume now that $\alpha\not\equiv\beta\pmod4$, then we have
\begin{equation*}
c_2(V_2)=(\alpha,\beta)_2(\alpha,\gamma)_2(\beta,\gamma)_2=(\alpha\beta,\gamma)_2=(-1)^{(\gamma-1)/2}.
\end{equation*}
On the other hand,
$$(-1,-(a-2)(b-2)(c-2))_2=(-1,\gamma)_2=(-1)^{(\gamma-1)/2}.$$
Thus by Lemma \ref{Cassels}, $V_2$ is isotropic.

When $a\equiv2\pmod4$ and $b\equiv0\pmod4$. We assume first that $k>2$, then it is clear that
$\ord_2(l_n+v_{a,b,c})=3$. Assume now that $k=2$, let the notations be as above, since
\begin{equation*}
(a-4)^2(b-2)(c-2)+(a-2)(b-2)(c-4)^2=2^3\beta(\alpha\mu^2+\gamma\ve^2),
\end{equation*}
with the essentially same method as above, we have if $\alpha\mu^2+\gamma\ve^2\equiv 0\pmod8$, then $V_2$
is isotropic, if $\alpha\mu^2+\gamma\ve^2\not\equiv0\pmod8$, then $l_n+v_{a,b,c}$ has bounded divisibility at $2$.

In sum, we finish the case when $2\nmid (a,b,c)$. Now, we consider the case when $2\mid (a,b,c)$,
we divide the proof of this part into the following three cases.

{\it Case} 2.1. $a\equiv b\equiv2\pmod4$ and $c\equiv 0\pmod4$.

If $k>l$, then it is clear that $\ord_2(l_n+v_{a,b,c})=1+l$.

If $k=l$, we put $a-4=2\ve,\ b-4=2\eta$, since
\begin{equation*}
(a-4)^2(b-2)(c-2)+(a-2)(b-4)^2(c-2)=(c-2)(\alpha\eta^2+\beta\ve^2),
\end{equation*}
in the same way as the above cases, we have $V_2$ is either isotropic or $l_n+v_{a,b,c}$ has bounded divisibility at $2$.

{\it Case} 2.2. $a\equiv 2\pmod 4$ and $b\equiv c\equiv 0\pmod 4$.

It is easy to see $\ord_2(l_n+v_{a,b,c})=2$.

{\it Case} 2.3. $a\equiv b\equiv c\equiv 2\pmod4$.

We put $a-4=2\ve,\ b-4=2\eta,\ c-4=2\mu$.

If $k>l$, then it is clear that $\ord_2(l_n+v_{a,b,c})=l+m$.

If $k=l=m$, since $v_{a,b,c}=2^{2k}(\ve^2\beta\gamma+\eta^2\alpha\gamma+\mu^2\alpha\beta)$, we have
$\ord_2(l_n+v_{a,b,c})=2k$.

If $k=l>m$, assume first that $l-m>2$, then
\begin{equation*}
2^{-2}((a-4)^2(b-2)(c-2)+(a-2)(b-4)^2(c-2))=2^{k+m}\gamma(\alpha\eta^2+\beta\ve^2),
\end{equation*}
using the same method as the above cases, if $\alpha\eta^2+\beta\ve^2\equiv0\pmod8$, then $V_2$ is isotropic,
if $\alpha\eta^2+\beta\ve^2\not\equiv0\pmod8$, note that $l>m+2$, we have
$\ord_2(l_n+v_{a,b,c})=k+m+\ord_2(\alpha\eta^2+\beta\ve^2)$.

Assume now that $l-m=1$, then $v_{a,b,c}=2^{2k}(\gamma(\frac{\alpha\eta^2+\beta\ve^2}{2})+\alpha\beta\mu^2)$, if $8\nmid\gamma(\frac{\alpha\eta^2+\beta\ve^2}{2})+\alpha\beta\mu^2$,
then it is clear that $\ord_2(l_n+v_{a,b,c})<2k$ (note that $k\ge3$ in this case).
If $8\nmid\gamma(\frac{\alpha\eta^2+\beta\ve^2}{2})+\alpha\beta\mu^2$, this implies that $\alpha\equiv\beta\pmod4$, we set $\alpha-\beta\equiv 4s\pmod8$, where $s\in\{0,1\}$, then
one may easily verify that $\gamma\equiv 2s-\alpha\pmod4$, then we have
\begin{equation*}
c_2(V_2)=(2^k\alpha,2^k\beta)_2(2^k\alpha,2^{k-1}\gamma)_2(2^k\beta,2^{k-1}\gamma)_2=(-1)^s(-1)^{(\alpha-1)/2}.
\end{equation*}
On the other hand,
$$(-1,-(a-2)(b-2)(c-2))_2=(-1,-\gamma)_2=(-1)^{(\gamma+1)/2}=(-1)^s(-1)^{(\alpha-1)/2}.$$
Thus by Lemma \ref{Cassels}, $V_2$ is isotropic.

Finally, if $l-m=2$, assume first that $\alpha\equiv\beta\pmod 4$, since
\begin{equation*}
2^{-2}((a-4)^2(b-2)(c-2)+(a-2)(b-4)^2(c-2))=2^{2k-2}\gamma(\alpha\eta^2+\beta\ve^2),
\end{equation*}
note that $\alpha\eta^2+\beta\ve^2\equiv 2\pmod 4$, thus $\ord_2(l_n+v_{a,b,c})=2k-1$.
Assume now that $\alpha\not\equiv\beta\pmod 4$, then
\begin{equation*}
c_2(V_2)=(2^k\alpha,2^k\beta)_2(2^k\alpha,2^{k-2}\gamma)_2(2^k\beta,2^{k-2}\gamma)_2=(-1)^{(\gamma-1)/2}.
\end{equation*}
On the other hand,
$$(-1,-(a-2)(b-2)(c-2))_2=(-1)^{(\gamma-1)/2}.$$
Thus, by Lemma \ref{Cassels}, $V_2$ is isotropic.

In view of the above, we complete the proof.
\end{proof}
By Lemma \ref{isotropic odd} and Lemma \ref{isotropic 2}, for any finite prime $p$,
we always have either $V_p$ is isotropic or $l_n+v_{a,b,c}$ has bounded divisibility at $p$.

\maketitle
\section{Proof of the main results}
\setcounter{lemma}{0}
\setcounter{theorem}{0}
\setcounter{corollary}{0}
\setcounter{remark}{0}
\setcounter{equation}{0}
\setcounter{conjecture}{0}

In this section, with the help of the lemmas in Section 2, we shall prove the main results of this paper.

\noindent{\it Proof of Theorem 1.1}. (1) By the decomposition of $\theta_{L+v}(z)$ in (\ref{decomposition}), if we can show that the $(l_n+v_{a,b,c})$-th fourier coefficient of
$\mathcal{U}(z)$ equals zero for each $n\in\N$, then we can easily verify that
$\mathcal{F}_{a,b,c}(x,y,z)$ is almost universal.

Since $\mathcal{U}(z)$ is a linear combination of finitely many unary congruence theta functions defined in
(\ref{unary}), thus the possible $u_{h,t}(z)$ in the decomposition with non-zero coefficient must satisfy
\begin{equation}
l_n+v_{a,b,c}=tr^2\equiv0\pmod t.
\end{equation}
Hence, (1) of Theorem \ref{Thm1.1} holds.

(2) Since the greatest common divisor of any two of $a-2,b-2,c-2$ is of the form $2^r$ with $r\in\N$.
Hence, for each odd prime $p$ dividing $(a-2)(b-2)(c-2)$, it is
clear that $p\nmid l_n+v_{a,b,c}$, thus the possible $u_{h,t}(z)$ in the decomposition with non-zero coefficient is of the form $u_{h,1}(z)$ or $u_{h,2}(z)$. Moreover,
if (i) and (ii) are satisfied, one may easily verify that $l_n+v_{a,b,c}\not\in\{r^2,2r^2: r\in\N\}$.
In sum, we have the $(l_n+v_{a,b,c})$-th fourier coefficient of
$\mathcal{U}(z)$ equals zero for each $n\in\N$, that is, $\mathcal{F}_{a,b,c}(x,y,z)$ is almost universal.

In view of the above, we complete the proof.
\qed
\medskip

\noindent{\it Proof of Proposition 1.1}. (1) As in the proof of Theorem \ref{Thm1.1}, it is sufficient to show
that $$(\mathcal{N}_{a,b,c}\cap\mathcal{S}_{a,b,c}^1)\cup(\mathcal{N}_{a,b,c}\cap\mathcal{S}_{a,b,c}^2)=\emptyset.$$
Since $a,b,c$ are distinct modulo $3$, it is easy to see that
$l_n+v_{a,b,c}\equiv 2\pmod3$, this implies that $l_n+v_{a,b,c}$ is not a square. On the other hand, if (i) is satisfied, then $l_n+v_{a,b,c}$ is odd, if (ii) or (iii) is satisfied, then $l_n+v_{a,b,c}\equiv 4\pmod8$,
this implies that $l_n+v_{a,b,c}$ is not of the form $2r^2$ with $r\in\N$. In sum, we have the $(l_n+v_{a,b,c})$-th fourier coefficient of
$\mathcal{U}(z)$ equals zero for each $n\in\N$, that is, $\mathcal{F}_{a,b,c}(x,y,z)$ is almost universal.

(2) As in the proof of (1), we have $l_n+v_{a,b,c}$ is not of the form $2r^2$ with $r\in\N$.
If $n\equiv a+1\pmod3$, then $l_n+v_{a,b,c}\equiv 2\pmod3$, this implies that $l_n+v_{a,b,c}$ is not a square.
Thus the $(l_n+v_{a,b,c})$-th fourier coefficient of
$\mathcal{U}(z)$ equals zero for each positive integer $n$ with $n\equiv a+1\pmod 3$, then it is clear that
(2) is true.
\qed
\medskip

\noindent{\it Proof of Corollary \ref{m,m+1,m+2}}.
By (2) of Theorem \ref{Thm1.1} and (1) of Proposition \ref{Pro1.1}, (1) of Corollary \ref{m,m+1,m+2} holds.
For odd prime $p$ dividing $m(m-1)(m-2)$, it is clear that $p\nmid l_n+v_{m,m+1,m+2}$ for each $n\in\N$,
thus the possible $u_{h,t}(z)$ in the decomposition with non-zero coefficient is of the form $u_{h,1}(z)$ or $u_{h,2}(z)$. Note that $m,m+1,m+2$ are distinct modulo $3$, thus with the same method in the proof of (1) of Proposition \ref{Pro1.1}, we see that the possible $u_{h,t}(z)$ in the decomposition with non-zero coefficient must be of the form $u_{h,2}(z)$. This implies (2) of Corollary \ref{m,m+1,m+2}.

In view of the above, we complete the proof.\qed
\medskip

\noindent{\it Proof of Corollary \ref{2^k}}. Let the notations be as in the introduction.
As in the proof of Theorem \ref{Thm1.1}, it is sufficient to show
that the $(l_n+v_{a,b,c})$-th fourier coefficient of $\mathcal{U}(z)$ equals zero for each $n\in\N$,
where $a=2^k\alpha+2$, $b=2^l\beta+2$ and $c=2^m\gamma+2$.
Since $\alpha,\beta,\gamma$ are pairwisely coprime, thus it is easy to see that for each odd
prime $p$ dividing $(a-2)(b-2)(c-2)$, $(p,l_n+v_{a,b,c})=1$. Hence, the possible $u_{h,t}(z)$ in the decomposition with non-zero coefficient is of the form $u_{h,1}(z)$ or $u_{h,2}(z)$.

(1) If the condition (i) is satisfied, then
\begin{equation*}
l_n+v_{a,b,c}=2^{3k+1}\alpha\beta\gamma n+2^{2k}(\alpha\gamma\eta^2+\beta\gamma\ve^2+\alpha\beta\mu^2).
\end{equation*}
Clearly, $\ord_2(l_n+v_{a,b,c})=2k$, thus $l_n+v_{a,b,c}\ne2r^2$ for any $r\in\N$. Moreover, since
$\alpha,\beta,\gamma$ are distinct modulo $3$, then $l_n+v_{a,b,c}\equiv-1\pmod3$, this implies that
$l_n+v_{a,b,c}\ne r^2$ for any $r\in\N$. In sum, we have the $(l_n+v_{a,b,c})$-th fourier coefficient of $\mathcal{U}(z)$ equals zero for each $n\in\N$.

(2) If (ii) is satisfied, then
\begin{equation*}
l_n+v_{a,b,c}=2^{k+m+l+1}\alpha\beta\gamma n+2^{l+m}(2^{k-l}\alpha\gamma\eta^2+\beta\gamma\ve^2+2^{k-m}\alpha\beta\mu^2).
\end{equation*}
Since $k>l$ and $k\equiv l\equiv m\pmod2$, thus $\ord_2(l_n+v_{a,b,c})$ is even,
then $l_n+v_{a,b,c}\ne2r^2$ for any $r\in\N$. Furthermore, since
$\alpha,\beta,\gamma$ are distinct modulo $3$, then $l_n+v_{a,b,c}\equiv-1\pmod3$, this implies that
$l_n+v_{a,b,c}\ne r^2$ for any $r\in\N$. In sum, we have the $(l_n+v_{a,b,c})$-th fourier coefficient of $\mathcal{U}(z)$ equals zero for each $n\in\N$.

(3) If (iii) is satisfied, then
\begin{equation*}
l_n+v_{a,b,c}=2^{2k+m+1}\alpha\beta\gamma n+2^{l+m}(\alpha\gamma\eta^2+\beta\gamma\ve^2+2^{k-m}\alpha\beta\mu^2).
\end{equation*}
Since $\alpha\equiv\beta\pmod 4$, using the condition that $k\not\equiv m\pmod 2$ and that $k-m>1$,
one can easily verify that
$\ord_2(l_n+v_{a,b,c})=k+m+1\equiv 0\pmod 2$, then $l_n+v_{a,b,c}\ne2r^2$ for any $r\in\N$. In addition,
Since $3\mid\gamma$ and $\alpha\equiv\beta\not\equiv 0\pmod 3$, we have $l_n+v_{a,b,c}\equiv-1\pmod3$, this implies that
$l_n+v_{a,b,c}\ne r^2$ for any $r\in\N$. In sum, we have the $(l_n+v_{a,b,c})$-th fourier coefficient of $\mathcal{U}(z)$ equals zero for each $n\in\N$.

In view of the above, we complete the proof.\qed

\end{document}